\documentclass[a4paper,10pt]{article}
\usepackage{amsmath,amsthm,amssymb}

\numberwithin{equation}{section}

\newcommand{\ot}{\otimes}
\newcommand{\id}{\text{id}}

\newcommand{\R}{\mathbb{R}}

\newcommand{\C}{\mathbb{C}}

\newcommand{\G}{\mathbb{G}}
\newcommand{\irr}{{\rm Irr}}

\newcommand{\HNN}{{\rm HNN}}

\theoremstyle{plain}
\newtheorem{theorem}{Theorem}[section]

\newtheorem{lemma}[theorem]{Lemma}
\newtheorem{proposition}[theorem]{Proposition}

\theoremstyle{definition}
\newtheorem{definition}[theorem]{Definition}
\newtheorem{example}[theorem]{Example}

\newtheorem{remark}[theorem]{Remark}

\begin{document}
\begin{center}
{\LARGE\bf K-amenability of HNN extensions of amenable discrete quantum groups}

\bigskip

{\sc Pierre Fima$^{(1,2)}$
\setcounter{footnote}{1}\footnotetext{Partially supported by the \textit{Agence Nationale de la Recherche} (Grant ANR 2011 BS01 008 01).}
\setcounter{footnote}{2}\footnotetext{Universit\'e Denis-Diderot Paris 7, IMJ, Chevaleret.
    \\ E-mail: pfima@math.jussieu.fr}}

\end{center}

\begin{abstract}
\noindent We construct the HNN extension of discrete quantum groups, we study their representation theory and we show that an HNN extension of amenable discrete quantum groups is $K$-amenable.\end{abstract}

\section{Introduction}
The notion of $K$-amenability for discrete groups was introduced by Cuntz \cite{Cu83} in order to give a simpler proof of a result of Pimsner and Voiculescu \cite{PV82} calculating the $K$-theory of the reduced $C^*$-algebra of a free group. Cuntz proved that the free product of $K$-amenable discrete groups is $K$-amenable. Julg and Valette \cite{JV84} extended the notion of $K$-amenability to the locally compact case and proved the $K$-amenability of locally compact groups acting on trees with amenable stabilizers. By Bass-Serre theory \cite{Se83}, this includes the case of amalgamated free products and HNN extensions of amenable discrete groups. Then, Pimsner \cite{Pi86} proved the $K$-amenability of locally compact groups acting on trees with $K$-amenable stabilizers.

At the quantum side, Skandalis \cite{Sk88} defined a notion of $K$-theoretic nuclearity for $C^*$-algebras analogous to Cuntz's $K$-theoritic amenability and Germain \cite{Ge96} proved that the free product of unital separable $K$-nuclear $C^*$-algebras in $K$-nuclear.

In the 1980's, Woronowicz \cite{Wo87}, \cite{Wo88}, \cite{Wo95} introduced the notion of compact quantum groups and generalized the classical Peter-Weyl representation theory. In this paper we consider discrete quantum groups as dual of compact quantum groups. Wang \cite{Wa95} introduced the amalgamated free product construction in the setting of Woronowicz and the free orthogonal and unitary quantum groups. The construction of free orthogonal and unitary quantum groups was generalized by Van Daele and Wang \cite{VW96}.

Baaj and Skandalis developed \cite{BS89} the equivariant $KK$-theory with respect to coactions of Hopf $C^*$-algebras and the general theory of locally compact quantum groups was done by Kustermans and Vaes \cite{KV00}. Vergnioux \cite{Ve04} developed the equivariant $KK$-theory for locally compact quantum groups and proved the $K$-amenability of amalgamated free products of discrete amenable quantum groups. Voigt \cite{Vo11} proved the $K$-amenability of free orthogonal quantum groups and Vergnioux and Voigt \cite{VV11} proved the $K$-amenability of free products of free orthogonal and unitary quantum groups.

The goal of this paper is to prove $K$-amenability of HNN extensions of discrete amenable quantum groups. The HNN construction of a given group $H$ is a group $\Gamma$ in which $H$ embeds in such a way that two given isomorphic subgroups of $H$ are conjugate. More precisely, given a subgroup $\Sigma< H$ and an injective homomorphism $\theta\,:\,\Sigma\rightarrow H$, the HNN extension is defined by $\Gamma=\langle H,t\,:\,\theta(\sigma)=t\sigma t^{-1}\,\,\forall\sigma\in\Sigma\rangle$. The name HNN is given in honor to G. Higman, B. H. Neuman and H. Neumann who were the first authors to consider this construction in \cite{HNN49}. This construction was developed by Ueda \cite{Ue05} in the setting of von Neumann algebras and $C^*$-algebras. Another approach was given by the author and S. Vaes \cite{FV12} in the setting of tracial von Neumann algebras. In this paper, we follow this approach to construct the HNN extension of discrete quantum groups, we study its representation theory and we prove the $K$-amenability in the case where the given starting quantum group is amenable.

This paper is organized as follows. The section $2$ is a preliminary section in which we fix some notations and recall some basic definitions and results about quantum groups and $K$-amenability. In section $3$ we give a detailed description of HNN extensions of $C^*$-algebras. In section $5$ we construct HNN extensions of discrete quantum groups and study their representation theory. Finally, we proved the $K$-amenability of HNN extensions of discrete amenable quantum groups in section $5$.

\section{Preliminaries}

All $C^*$-algebras are supposed to be separable and unital and all Hilbert $C^*$-modules are supposed to be separable. Let $A$ be a $C^*$-algebra and $H$ a Hilbert $A$-module. The $A$-valued scalar product is denoted by $\langle .,.\rangle$ and is supposed to be linear in the second variable. The $C^*$-algebra of adjointable maps on $H$ is denoted by $\mathcal{L}_A(H)$. We will use the same symbol $\ot$ to denote the tensor product of Hilbert $C^*$-modules and the minimal tensor product of $C^{*}$-algebras. We use the symbol $\odot$ to denote the algebraic tensor product of vector spaces. We will use freely the leg numbering notation.

\begin{definition}[Woronowicz]
A \textit{compact quantum group is a pair} $\G=(A,\Delta)$, where $A$ is a unital $C^{*}$-algebra, $\Delta$ is unital *-homomorphism from $A$ to $A\ot A$ satisfying $(\Delta\ot\id)\Delta=(\id\ot\Delta)\Delta$ and $\Delta(A)(A\ot 1)$ and $\Delta(A)(1\ot A)$ are dense in $A\ot A$.
\end{definition}
We denote by $C(\G)$ the $C^{*}$-algebra $A$. The major results in the general theory of compact quantum groups are the existence and uniqueness of the Haar state and the Peter-Weyl representation theory.
\begin{theorem}[Woronowicz]
Let $\G$ be a compact quantum group. There exists a unique state $\varphi$ on $C(\G)$ such that $(\id\ot\varphi)\Delta(a)=\varphi(a)1=(\varphi\ot\id)\Delta(a)$ for all $a\in C(\G)$. The state $\varphi$ is called the {\rm Haar state} of $\G$.
\end{theorem}
The Haar state need not be faithful. When the Haar state is faithful we say that $\G$ is \textit{reduced}. Let $C_{\text{red}}(\G):=C(\G) / I$ be the reduced $C^*$-algebra of $\G$, where $I=\{x\in A\,|\,\varphi(x^{*}x)=0\}$.  $C_{\text{red}}(\G)$ has a canonical structure of compact quantum group, called the \textit{reduced compact quantum group of} $\G$.

A unitary representation of dimension $n$ of $\G$ is a unitary $u\in M_n(\C)\ot C(\G)$ such that $(\id\ot\Delta)(u)=u_{12}u_{13}$. If $u\in M_n(\C)\ot C(\G)$ and $v\in M_k(\C)\ot C(\G)$ we define their \textit{tensor product} by
$$u\ot v=u_{13}v_{23}\in M_n(\C)\ot M_k(\C)\ot C(\G).$$ An \textit{intertwiner} between $u$ and $v$ is a linear map $T\,:\,\C^n\rightarrow \C^k$ such that $(T\ot 1)u=v(T\ot 1)$. The unitary representations $u$ and $v$ are called \textit{unitarily equivalent} if there exists a unitary intertwiner between $u$ and $v$. We call $u$ \textit{irreducible} if the only intertwiners between $u$ and $u$ are the scalar multiples of the identity.

\begin{theorem}[Woronowicz]
Every unitary representation is unitarily equivalent to a direct sum of irreducible unitary representations.
\end{theorem}

We denote by $\irr(\G)$ the set of (equivalence classes) of irreducible unitary representations of a compact quantum group $\G$. For each $x\in\irr(\G)$ we choose a representative $u^{x}\in M_{n_x}\ot C(\G)$. The class of the trivial representation is denoted by $1$.

We denote by $\mathcal{C}(\G)$ the linear span of the coefficients of the $u^x$ for $x\in\irr(\G)$. It is a unital dense *-subalgebra of $C(\G)$. Let $C_{\text{max}}(\G)$ be the maximal $C^{*}$-completion of the unital *-algebra $\mathcal{C}(\G)$.  $C_{\text{max}}(\G)$ has a canonical structure of a compact quantum group called the \textit{maximal quantum group} of $\G$. Observe that we have a canonical surjective morphism $\lambda\,:\,C_{\text{max}}(\G)\rightarrow C_{\text{red}}(\G)$ which is the identity on $\mathcal{C}(\G)$. $\G$ is called \textit{amenable} if $\lambda$ is an isomorphism. $\G$ is called $K$\textit{-amenable} if there exists $\alpha\in {\rm KK}(C_{\text{red}}(\G),\C)$ such that $\lambda^*(\alpha)=[\epsilon]\in{\rm KK}(C_{\text{max}}(\G),\C)$ where $\epsilon\,:\,C_{\text{max}}(\G)\rightarrow\C$ is the trivial representation i.e., $(\id\ot\epsilon)(u^x)=1$ for all $x\in\irr(\G)$.

\section{{\rm HNN} extensions of $C^*$-algebras.}\label{HNNCstar}

\textbf{The reduced HNN extension}
 
The reduced HNN extension was introduced in \cite{Ue05}. Here, we follow the approach of \cite{FV12}.

Let $B\subset A$ be a unital $C^*$-subalgebra of the unital $C^*$-algebra $A$ and $\theta\,:\,A\rightarrow B$ be an injective $*$-homomorphism. Define, for $\epsilon\in\{-1,1\}$, 
$$B_{\epsilon}=\left\{\begin{array}{lcl}
B&\text{if}&\epsilon=1,\\
\theta(B)&\text{if}&\epsilon=-1.
\end{array}\right.$$
We define $\theta^{\epsilon}\,:B_{\epsilon}\rightarrow B_{-\epsilon}\subset A$ in the obvious way.

We suppose that there exist conditional expectations $E_\epsilon\,:A\rightarrow B_{\epsilon}$ for $\epsilon\in\{-1,1\}$. For $\epsilon=1$, we denote by $(H_{1},\pi_{1},\eta_{1})$ the G.N.S. construction associated to $E_{1}$ i.e. $H_{1}$ is the Hilbert $B$-module obtained by separation and completion of $A$ for the $B$-valued scalar product $\langle x,y\rangle= E_{1}(x^*y)$, $x,y\in A$, the right action of $B$ is given by right multiplication, $\pi_{1}$ is the representation of $A$ on $H_{1}$ given by left multiplication and $\eta_{1}$ is the image of $1$ in $H_{1}$. For $\epsilon=-1$, we denote by $(H_{-1},\pi_{-1},\eta_{-1})$ the ``G.N.S. construction'' associated to $\theta^{-1}\circ E_{-1}$ i.e. $H_{-1}$ is the Hilbert $B$-module obtained by separation and completion of $A$ for the $B$-valued scalar product $\langle x,y\rangle= \theta^{-1}\circ E_{-1}(x^*y)$, $x,y\in A$, the right action of $b\in B$ is given by the right multiplication by $\theta(b)$, $\pi_{-1}$ is the representation of $A$ on $H_{-1}$ given by left multiplication and $\eta_{-1}$ is the image of $1$ in $H_{-1}$.

Observe that, for $\epsilon\in\{-1,1\}$, the map $(b\mapsto \pi_{\epsilon}(b)\eta_{\epsilon})$ is faithful on $B_{\epsilon}$ (hence $\pi_{\epsilon}|_{B_{\epsilon}}$ is also faithful). Although the representation $\pi_{\epsilon}$ may be not faithful on $A$ we will simply write $a\xi$ for $\pi_{\epsilon}(a)\xi$ when $\xi\in H_{\epsilon}$ and $a\in A$. We will also use the notation $\widehat{a}=\pi_{\epsilon}(a)\eta_{\epsilon}\in H_{\epsilon}$ for $a\in A$.

Observe that the submodule $\eta_{\epsilon} B$ is orthogonally complemented in $H_{\epsilon}$. Denote by $H_{\epsilon}^{\circ}$ the orthogonal complement of $\eta_{\epsilon} B$ in $H_{\epsilon}$ (it is the closure of $\{x\eta_{\epsilon}\,:\,E_{\epsilon}(x)=0\}$). One has $H_{\epsilon}=\eta_{\epsilon} B\oplus H_{\epsilon}^{\circ}$ and $B_{\epsilon}H_{\epsilon}^{\circ}= H_{\epsilon}^{\circ}$. 

For $n\geq 1$ and $\epsilon_1,\ldots,\epsilon_n\in\{-1,1\}$ define $K_0=H_{-\epsilon_1}$, $K_n=H_1$ and, for $n\geq 2$ and $1\leq i\leq n-1$,
$$K_i=\left\{\begin{array}{lcl}
H_{-\epsilon_i} &\text{if}&\epsilon_i=\epsilon_{i+1},\\
H_{\epsilon_i}^{\circ}&\text{if}&\epsilon_i\neq\epsilon_{i+1}.\end{array}\right.$$

For $i=0,\ldots,n$, we view all the $K_i$ as a Hilbert $B$-module as explained before. For $i=1,\ldots,n$ we have a representation $\rho_i\,:\,B\rightarrow \mathcal{L}_B(K_i)$ defined by, if $\xi\in K_i$ and $b\in B$,
$$\rho_i(b)\xi=\left\{\begin{array}{lcl}
b\xi&\text{if}&\epsilon_i=1,\\
\theta(b)\xi&\text{if}&\epsilon_i=-1.\end{array}\right.$$

Define the Hilbert $B$-module $\mathcal{H}_{\epsilon_1,\ldots,\epsilon_n}=K_0\underset{\rho_1}{\ot}\ldots\underset{\rho_n}{\ot} K_n$. The left action of $A$ on $K_0$ by left multiplication induces a left action of $A$ on $\mathcal{H}_{\epsilon_1,\ldots,\epsilon_n}$ in the obvious way.

We define the Hilbert $B$-module $\mathcal{H}$ by the orthogonal direct sum
$$\mathcal{H}=H_1\oplus\bigoplus_{n\geq 1,\,\epsilon_1,\ldots,\epsilon_n\in\{-1,1\}}\mathcal{H}_{\epsilon_1,\ldots,\epsilon_n},$$
with the left action of $A$ given by the direct sum of the left actions of $A$ on the Hilbert $B$-modules $\mathcal{H}_{\epsilon_1,\ldots,\epsilon_n}$ and the left action of $A$ on $H_1$. We denote this action by $\pi\,:\, A\rightarrow \mathcal{L}_B(\mathcal{H})$.

Observe that $\pi|_B$ is faithful. Also, if $\pi_1$ is faithful then $\pi$ is faithful.

Let $\epsilon\in\{-1,1\}$. We define an operator $u^{\epsilon}$ on $\mathcal{H}$ in the following way.
\begin{itemize}
\item If $\xi\in H_1$ we define $u^{\epsilon}\xi=\widehat{1}\ot\xi\in\mathcal{H}_{\epsilon}$.
\item If $\xi\in \mathcal{H}_{\epsilon_1,\ldots,\epsilon_n}$ with $n\geq 1$ and $\epsilon_1=\epsilon$ we define $u^{\epsilon}\xi=\widehat{1}\ot\xi\in \mathcal{H}_{\epsilon,\epsilon_1,\ldots,\epsilon_n}$.
\item If $\xi=\widehat{a}\ot\xi_0\in\mathcal{H}_{\epsilon_1}$ with $\epsilon_1\neq\epsilon$ and $\widehat{a}\in K_0=H_{\epsilon}$, $\xi_0\in K_1=H_1$ we define
$$u^{\epsilon}(\widehat{a}\ot\xi_0)=\left\{
\begin{array}{llcl}
\widehat{1}\ot\widehat{a}\ot\xi_0&\in\mathcal{H}_{\epsilon,\epsilon_1}&\text{if}&E_{\epsilon}(a)=0,\\
\theta^{\epsilon}(a)\xi_0&\in H_{1}&\text{if}&a\in B_{\epsilon}.\\
\end{array}\right.$$
\item If $\xi=\widehat{a}\ot\xi_0\in\mathcal{H}_{\epsilon_1,\ldots,\epsilon_n}$ with $n\geq 2$, $\epsilon_1\neq\epsilon$ and $\widehat{a}\in K_0$, $\xi_0\in K_1\underset{\rho_2}{\ot}\ldots\underset{\rho_n}{\ot} K_n$ (which is a sub-$B$-module of $\mathcal{H}_{\epsilon_2,\ldots,\epsilon_n}$) we define
$$u^{\epsilon}(\widehat{a}\ot\xi_0)=\left\{
\begin{array}{llcl}
\widehat{1}\ot\widehat{a}\ot\xi_0&\in\mathcal{H}_{\epsilon,\epsilon_1,\ldots,\epsilon_n}&\text{if}&E_{\epsilon}(a)=0,\\
\theta^{\epsilon}(a)\xi_0&\in\mathcal{H}_{\epsilon_2,\ldots,\epsilon_n}&\text{if}&a\in B_{\epsilon}.\\
\end{array}\right.$$
\end{itemize}
It is easy to check that $u^{\epsilon}$ commutes with the right action of $B$ and extends to a unitary on the Hilbert $C^*$-module $\mathcal{H}$ such that $(u^{\epsilon})^*=u^{-\epsilon}$ so that the superscript $\epsilon$ really means ``to the power $\epsilon$''. We denote by $u$ the unitary $u^1$. One can also easily check the following formula :
$$u\pi(b)u^*=\pi(\theta(b))\quad\text{for all}\quad b\in B.$$

Although it is not necessary, we will assume, to simplify notations and for the rest of this section, that $E_{\epsilon}$, for $\epsilon\in\{-1,1\}$, is G.N.S. faithful i.e., $\pi_{\epsilon}$ is faithful. Hence, $\pi$ is faithful and we may and will assume that $A\subset\mathcal{L}_B(\mathcal{H})$ and $\pi=\id$. The preceding relation becomes $ubu^*=\theta(b)$ for all $b\in B$.

\begin{definition}
The \textit{reduced} HNN \textit{extension} $\HNN(A,B,\theta)$ is the $C^*$-subalgebra of $\mathcal{L}_B(\mathcal{H})$ generated by $A$ and $u$:
$$\HNN(A,B,\theta):=\langle A,u\rangle\subset\mathcal{L}_B(\mathcal{H}).$$
\end{definition}

Let $P=\HNN(A,B,\theta)$. An operator $x\in P$ of the form $x=x_0u^{\epsilon_1}x_1\ldots u^{\epsilon_n}x_n$ with $n\geq 1$, $x_i\in A$ and $\epsilon_i\in\{-1,1\}$ will be called \textit{reduced} if for all $1\leq i\leq n-1$ we have $E_{\epsilon_i}(x_i)=0$ whenever $\epsilon_{i+1}\neq\epsilon_i$. Observe that our terminology is different from the one adopt in \cite{FV12}: we do not allow $n=0$ in the definition of a reduced operator.

Let $\Omega=\eta_1\in H_1\subset\mathcal{H}$. Observe that $\Omega$ is $B$-central. Namely, $b\Omega=\Omega b$ for all $b\in B$. Let $x=x_0u^{\epsilon_1}\ldots u^{\epsilon_n}x_n$ be a reduced operator. One has
\begin{eqnarray}\label{gns}
x\Omega=\hat{x}_0\ot\ldots\ot\hat{x}_n\in\mathcal{H}_{\epsilon_1,\ldots,\epsilon_n}.
\end{eqnarray}
It follows that the integer $n$ (and the sequence $\epsilon_1,\ldots,\epsilon_n$) only depends on the operator $x$. The integer $n$ is called the \textit{length} of the reduced operator $x$.

Let $\mathcal{P}$ be the vector subspace of $P$ spanned by the reduced operators and $A$. By the relation $\theta(b)=ubu^*$ for $b\in B$, it is easy to check that $\mathcal{P}$ is a *-subalgebra of $P$. Moreover, by definition of the HNN extension, $\mathcal{P}$ is dense in $P$.

Define, for $x\in P$, $E_B(x)=\langle\Omega,x\Omega\rangle\in B$. It is easily seen that $E_B$ a conditional expectation onto $B$ satisfying $E_B|_A=E_1$. Moreover, using $(\ref{gns})$, we see that, for all reduced operator $x\in P$, one has $E_B(x)=0$. Equation $(\ref{gns})$ also implies that $\overline{P\Omega}=\mathcal{H}$ hence, $(\mathcal{H},\id,\Omega)$ is the GNS construction of $E_B$.

Define, for $x\in P$, $E_{\theta(B)}(x)=uE_B(u^*xu)u^*\in \theta(B)$. Again, it is easy to check that $E_{\theta(B)}$ a conditional expectation onto $\theta(B)$ satisfying $E_{\theta(B)}|_A=E_{-1}$.

The reduced HNN extension $P$ satisfies the following universal property.

\begin{proposition}\label{universal}
Let $C$ be a unital $C^*$-algebra with a unital faithful $*$-homomorphism $\rho\,:\,A\rightarrow C$. Suppose that there exists a unitary $w\in C$ and a conditional expectation $E'$ from $C$ to $\rho(B)$ such that:
\begin{enumerate}
\item $C$ is generated by $\rho(A)$ and $w$.
\item $w\rho(b)w^*=\rho(\theta(b))$ for all $b\in B$ and $E'\circ\rho=\rho\circ E_1$.
\item For all $n\geq 1$, $\epsilon_1,\ldots,\epsilon_n\in\{-1,1\}$ one has $E'(\rho(x_0)w^{\epsilon_1}\ldots w^{\epsilon_n}\rho(x_n))=0$ for all $x_i\in A$ such that $E_{\epsilon_i}(x_i)=0$ whenever $\epsilon_i\neq \epsilon_{i+1}$, $1\leq i\leq n-1$.
\item $E'$ is G.N.S. faithful i.e., for all $x\in C$, if $E'(y^*x^*xy)=0$ for all $y\in C$, then $x=0$.
\end{enumerate}
Then,  there exists a unique $*$-isomorphism $\widetilde{\rho}\,:\, P\rightarrow C$ such that
$$\widetilde{\rho}(u)=w\quad\text{and}\quad\widetilde{\rho}(a)=\rho(a)\,\,\text{for all}\,\,a\in A.$$
Moreover, $\widetilde{\rho}$ intertwines $E'$ and $E_B$.
\end{proposition}

\begin{proof}
Since $P$ is generated by $A$ and $u$, the uniqueness is obvious. Let $(H',\rho',\eta')$ be the GNS construction of $\rho^{-1}\circ E'$ i.e, $H'$ as a Hilbert $B$-module obtained by separation and completion of $C$ for the $B$-valued scalar product $\langle x,y\rangle= \rho^{-1}\circ E'(x^*y)$, the right action of $b\in B$ is given by the right multiplication by $\rho(b)$, $\rho'$ is the representation of $C$ given by left multiplication and $\eta'$ is the image of $1$ in $H'$. By $4$, $\rho'$ is faithful so we may and will assume that $C\subset\mathcal{L}_B(H')$ and $\rho'=\id$. Define
$V\,:\,\mathcal{H}\rightarrow H'$ by $Va\Omega=\rho(a)\eta'$ for $a\in A$ and, for $x=x_0u^{\epsilon_1}\ldots u^{\epsilon_n}x_n\in P$ a reduced operator, $Vx\Omega=\rho(x_0)w^{\epsilon_1}\ldots w^{\epsilon_n}\rho(x_n)\eta'$. It is easy to check that $V$ extends to a unitary $V\in\mathcal{L}_B(\mathcal{H},H')$ such that $VaV^*=\rho(a)$ for all $a\in A$ and $VuV^*=w$. Then, $\widetilde{\rho}(x)=VxV^*$ does the job. 
\end{proof}

We construct now a conditional expectation from $P$ to $A$. Let $Q\in\mathcal{L}_B(\mathcal{H})$ be the projection onto the Hilbert sub-B-module $H_1$ of $\mathcal{H}$. Then, it is easy to check that the formula $E_A(x)=QxQ\in\mathcal{L}_B(Q\mathcal{H})=\mathcal{L}_B(H_1)$ defines a conditional expectation from $P$ to $A\subset\mathcal{L}_B(H_1)$ satisfying:
$$E_A(x)=0\quad\text{for all reduced operator $x\in P$}.$$
Moreover, $E_{\epsilon}\circ E_A=E_{B_{\epsilon}}$ for $\epsilon\in\{-1,1\}$.

\textbf{The maximal HNN extension}

The maximal (or full, or universal) HNN extension was also introduced in \cite{Ue05}. We keep the same notations as before and we still assume that we have conditional expectations with faithful G.N.S. constructions from $A$ to $B_{\epsilon}$ for $\epsilon\in\{-1,1\}$. The maximal HNN extension is the unital $C^*$-algebra $P_m$ generated by $A$ and a unitary $w\in P_m$ such that  $wbw^*=\theta(b)$ for all $b\in B$ and satisfying the universal property that whenever $C$ is a unital $C^*$-algebra with a unitary $u\in C$ and a $*$-homomorphism $\rho\,:\,A\rightarrow C$ such that $u\rho(b)u^*=\rho(\theta(b))$ for all $b\in B$ there exists a unique $*$-homomorphism $\widetilde{\rho}\,:\,P_m\rightarrow C$ such that $\widetilde{\rho}|_A=\rho$ and $\widetilde{\rho}(w)=u$. Such a $C^*$-algebra is obviously unique (up to a canonical isomorphism) and is denoted by ${\rm HNN}_{\text{max}}(A,B,\theta)$.

\section{HNN extensions of Compact Quantum Groups}\label{HNNCQG}

We consider two reduced compact quantum groups $\G_A=(A,\Delta_A)$ and $\mathbb{G}_B=(B,\Delta_B)$. We denote by $\varphi_A$ and $\varphi_B$ the Haar (faithful) states on $A$ and $B$ respectively.

We suppose that $\theta\,:\,B\rightarrow A$ is an injective unital *-homomorphism which intertwines the comultiplications.  Hence, $\theta(B)$ is a Woronowicz $C^*$-subalgebra of $A$. By \cite{Ve04}, $\varphi_A\circ\theta=\varphi_B$ and there exists a unique conditional expectation $E_{\theta}\,:\, A\rightarrow\theta(B)$ such that $\varphi_A=\varphi_B\circ\theta^{-1}\circ E_{\theta}$. Since $\varphi_A$ is faithful, $E_{\theta}$ is faithful. In particular, $E_{\theta}$ is G.N.S. faithful. This conditional expectation is also characterized by the following invariance property:
$$(\id\ot E_{\theta})\circ\Delta_A=(E_{\theta}\ot\id)\circ\Delta_A=\Delta_B\circ\theta^{-1}\circ E_{\theta}=\Delta_A\circ E_{\theta}.$$

We suppose that $B\subset A$ is a Woronowicz $C^*$-subalgebra. We can apply the preceding discussion to the map $\theta=\id$. In particular, we have a G.N.S. faithful conditional expectation $E_1\,:\,A\rightarrow B$. Let $\theta\,:\,B\rightarrow A$ be an embedding which intertwines the comultiplications. Once again, the preceding discussion applies to $\theta$ and we have a G.N.S. faithful conditional expectation $E_{-1}\,:\,A\rightarrow \theta(B)$. We will freely use the notations and results of section \ref{HNNCstar}. Define $P=\HNN_{\text{red}}(A,B,\theta)=\langle A, u\rangle$. 

From the hypothesis, we also have canonical inclusions $C_{\text{max}}(\G_B)\subset C_{\text{max}}(\G_A)$ which intertwine the comultiplications. Also, since the injective morphism $\theta\,:\,\mathcal{C}(\G_B)\rightarrow\mathcal{C}(\G_A)$ intertwines the comultiplications we have a canonical injective morphism $\theta\,:\,C_{\text{max}}(\G_B)\rightarrow C_{\text{max}}(\G_A)$ which intertwines the comultiplications.

Define $P_{m}:={\rm HNN}_{\text{max}}(C_{\text{max}}(\G_A),C_{\text{max}}(\G_B),\theta)=\langle C_{\text{max}}(\G_A), w\rangle$. By the universal property, there exists a unique $*$-homomorphism $\Delta_m\,:\, P_m\rightarrow P_m\ot P_m$ such that
$$\Delta(w)=w\ot w\quad\text{and}\quad\Delta_m(a)=\Delta_A(a)\,\,\,\,\forall a\in C_{\text{max}}(\G_A).$$
$P_m$ is generated, as a $C^*$-algebra, by the elements $v^x_{i,j}$ for $x\in{\rm Irr}(\G_A)$ and $1\leq i,j\leq \text{dim}(x)$ and by $w$ for which it is easy to check that the conditions of \cite[Definition 2.1']{Wa95} are satisfied. Hence, $\G_m=(P_m,\Delta_m)$ is a compact quantum group.

Let us denote by $\lambda$ the canonical surjective morphism from $C_{\text{max}}(\G_A)$ to $A$. By the universal property, we have a unique $*$-homomorphism, still denoted by $\lambda$, from $P_m$ to $P$ such that
$$\lambda(w)=u\quad\text{and}\quad\lambda(a)=\lambda(a)\,\,\,\text{for all}\,\,a\in C_{\text{max}}(\G_A).$$

We view $\irr(\G_B)$ as a subset of $\irr(\G_A)$ and we also view $\irr(\G_A)$ as a subset of $\irr(\G_m)$. The map $\theta$ induces an injective map, still denoted by $\theta$, from $\irr(\G_B)$ to $\irr(\G_A)$. For $\epsilon\in\{-1,1\}$ we define
$$\irr(\G_A)_{\epsilon}=\left\{\begin{array}{lcl}
\irr(\G_A)\setminus\irr(\G_B)&\text{if}&\epsilon=1,\\
\irr(\G_A)\setminus\theta(\irr(\G_B))&\text{if}&\epsilon=-1.\end{array}\right.$$
Observe that $w\in P_m$ is a irreducible representation of $\G_m$ of dimension $1$.

Let $v$ be a unitary representation of $\G_m$. We call $v$ \textit{reduced} if $v$ is of the form $v=v^{x_0}\ot w^{\epsilon_1}\ot\ldots\ot w^{\epsilon_n}\ot v^{x_n}$ where $n\geq 1$, $x_k\in\irr(\G_A)$ and $\epsilon_k\in\{-1,1\}$ are such that, for all $1\leq k\leq n-1$,  $x_k\in\irr(\G_A)_{\epsilon_k}$ whenever $\epsilon_k\neq\epsilon_{k+1}$.

\begin{theorem}\label{rep}
The following holds.
\begin{enumerate}
\item The Haar state is given by $\varphi_m=\varphi_A\circ E_A\circ\lambda$.
\item Every non-trivial irreducible unitary representation of $\G_m$ is unitarily equivalent to a subrepresentation of a reduced representation or to an irreducible representation of $\G_A$. Hence, $\irr(\G_A)$ and $w$ generate the representation category of $\G_m$.
\item The reduced $C^*$-algebra of $\G_m$ is $P$, the maximal one is $P_m$.
\end{enumerate}
\end{theorem}

\begin{proof}
$1.$ Let $\mathcal{P}_m\subset P_m$ be the linear span of the coefficients of the reduced representations. It is easy to see that the linear span of $\mathcal{P}_m$ and $A$ is a dense $*$-subalgebra of $P_m$. Moreover, $\Delta_m(\mathcal{P}_m)\subset\mathcal{P}_m\odot\mathcal{P}_m$ and $\lambda(\mathcal{P}_m)$ is contained in the linear span of the reduced operators in $P$. Hence, $E_A\circ\lambda(\mathcal{P}_m)=\{0\}$. It follows that, for all $x\in\mathcal{P}_m$, $(\id\ot\varphi_m)\Delta_m(x)=(\varphi_m\ot\id)\Delta_m(x)=0=\varphi_m(x)1$. Hence, it suffices to check the invariance property for $x$ a coefficient of a irreducible representation of $\G_A$ for which it is obvious.

$2.$ Since the linear span of the coefficients of the reduced representations and the coefficients of the irreducible representations of $\G_A$ is dense in $P_m$, the result follows from the general theory.

$3.$ Since the morphism $\lambda$ is surjective and the state $\varphi_A\circ E_A$ is faithful on $P$, it follows from $1$ that the reduced $C^*$-algebra of $\G_m$ is $P$. Moreover, it follows from $2$ that $\mathcal{C}(\G_m)$ is equal to the linear span of $\mathcal{P}_m$ and $\mathcal{C}(\G_A)$. Hence, $C_{\text{max}}(\G_m)$ is generated, as a $C^*$-algebra, by $\mathcal{C}(\G_A)$ and $w$. By the universal property of $C_{\text{max}}(\G_A)$, we have a $*$-homomorphism from $C_{\text{max}}(\G_A)$ to $C_{\text{max}}(\G_m)$ which is the identity on $\mathcal{C}(\G_A)$. Since the relation $\theta(b)=wbw^*$ holds in $C_{\text{max}}(\G_m)$ for all $b\in C_{\text{max}}(\G_B)$, we have a surjective homomorphism from the HNN extension $P_m$ to $C_{\text{max}}(\G_m)$ which is the identity of $\mathcal{C}(\G_m)$. It follows that  $P_m=C_{\text{max}}(\G_m)$.
\end{proof}

\begin{remark}
One could have have constructed first the reduced compact quantum group and prove that the maximal one is $\G_m$. Indeed, one can prove directly, at the reduced level, that there exists a unique $*$-homomorphism $\Delta\,:\,P\rightarrow P\otimes P$ such that 
$$\Delta(u)=u\otimes u\quad\text{and}\quad\Delta(x)=\Delta_A(x)\,\,\forall x\in A.$$
To prove that, it suffices to consider the $C^*$-subalgebra $C$ of $P\ot P$ generated by $\Delta_A(A)$ and $u\ot u$, to view $\rho=\Delta_A$ as a unital faithful $*$-homomorphism from $A$ to $C$, and to check the hypothesis of Proposition \ref{universal} with the conditional expectation $E'=(\id\ot E_1)|_C$. One can also check easily that $\varphi=\varphi_A\circ E_A$ is $\Delta$-invariant. It follows from the general theory that $(P,\Delta)$ is a reduced compact quantum group. Moreover, one can show that the maximal $C^*$-algebra of $(P,\Delta)$ is $P_m$ by studying the representations, as in the proof of Theorem \ref{rep}.
\end{remark}

\begin{example}
Let $N<G$ be a non-trivial closed normal subgroup of a compact group $G$ and define $K=G/N$. Let $\theta\,:\,G\rightarrow K$ be a continuous surjective group homomorphism. View $C(K)\subset C(G)$ as $N$-right invariant functions and define the injective $*$-homomorphism $C(K)\rightarrow C(G)$ by composing with $\theta$. Then, one can perform the HNN construction to get a compact quantum group which is non-commutative and non-cocommutative whenever $G$ is non-commutative.
\end{example}

\section{K-amenability}

This section contains the proof of the following theorem.

\begin{theorem}\label{Kamenability}
An HNN extension of amenable discrete quantum groups is K-amenable.
\end{theorem}

\begin{proof}

Let $P=\text{HNN}(A,B,\theta)=\langle A, u\rangle$ be a reduced HNN extension of $C^*$-algebras. Let $E_A$ and $E_B$ be the conditional expectations from $P$ to $A$ and $B$ respectively.

\begin{lemma}\label{iso}
Let $x=x_0u^{\epsilon_1}\ldots u^{\epsilon_n}\in P$ be a reduced word in $P$.
\begin{enumerate}
\item If $\epsilon_n=-1$ then $E_A(x^*x)=\theta\circ E_B((xu)^*xu)$.
\item If $\epsilon_n=1$ then $E_A(x^*x)=E_B(x^*x)$.
\end{enumerate}
\end{lemma}

\begin{proof}
$1.$ We prove it by induction on $n$. If $n=1$, write $x=x_0u^*$, then
$$E_A(x^*x)=E_A(ux_0^*x_0u^*)=E_A(uE_B(x_0^*x_0)u^*)=\theta\circ E_B(x_0^*x_0)=\theta\circ E_B((xu)^*xu).$$
Suppose that $1$ holds for $n$. Let $x=x_0u^{\epsilon_1}\ldots u^{\epsilon_n}x_nu^*$ be a reduced word. In the following computation we use the notation $E_{-1}=E_{\theta(B)}$ and $E_1=E_B$. One has
\begin{eqnarray*}
E_A(x^*x)&=&E_A(ux_n^*u^{-\epsilon_n}\ldots u^{-\epsilon_1}x_0^*x_0u^{\epsilon_1}\ldots u^{\epsilon_n}x_nu^*)\\
&=&E_A(ux_n^*u^{-\epsilon_n}\ldots \theta^{-\epsilon_1}(E_{-\epsilon_1}(x_0^*x_0))\ldots u^{\epsilon_n}x_nu^*)=E_A(y^*y),
\end{eqnarray*}
where $y=x_1'u^{\epsilon_2}\ldots u^{\epsilon_n}x_nu^*$ and $x_1'=\sqrt{\theta^{-\epsilon_1}(E_{-\epsilon_1}(x_0^*x_0))}x_1$. Observe that $y$ is reduced of length $n$ and ends with $u^*$. By the induction hypothesis one has $E_A(y^*y)=\theta\circ E_B((yu)^*yu)$. But
\begin{eqnarray*}
E_B((yu)^*yu)&=&E_B(x_n^*u^{-\epsilon_n}\ldots \theta^{-\epsilon_1}(E_{-\epsilon_1}(x_0^*x_0))\ldots u^{\epsilon_n}x_n)\\
&=&E_B(x_n^*u^{-\epsilon_n}\ldots u^{-\epsilon_1}x_0^*x_0u^{\epsilon_1}\ldots u^{\epsilon_n}x_n)=E_B((xu)^*xu).
\end{eqnarray*}
This finishes the proof of $1$. The proof of $2$ is similiar.
\end{proof}

We suppose that $\G_A=(A,\Delta_A)$ and $\G_B=(B,\Delta_B)$ are two reduced compact quantum groups such that the inclusion $B\subset A$ and the $*$-homomorphism $\theta$ intertwine the comultiplications. Assume that $\G_A$ is coamenable and let $\epsilon\,:\, A\rightarrow\C$ be the counit. Hence, $P$ is the reduced $C^*$-algebra of the HNN extension compact quantum group. Observe that $\epsilon\circ \theta=\epsilon$. Let $(H,\pi,\xi)$ be the GNS construction of the state $\epsilon\circ E_A$ on $P$ and $(K,\rho,\eta)$ be the GNS construction of the state $\epsilon\circ E_B$ on $P$. Define
$$H_0=\C\xi,\,\,H_{\pm 1}=\overline{\text{Span}}\{\pi(x)\xi\,:\,\,x=x_0u^{\epsilon_1}\ldots u^{\epsilon_n}\,\,\text{is a reduced word with}\,\epsilon_n=\pm 1\}.$$
Observe that the spaces $H_0, H_{-1}, H_1$ are pairwise orthogonal. Moreover, since $\pi(a)\xi=\epsilon(a)\xi$ for all $a\in A$, one has $H=H_0\oplus H_{-1}\oplus H_1$. We also define
$$
K_{-1}=\overline{\text{Span}}\{\rho(x)\eta\,:\,x\in A\,\,\text{or}\,\,x=x_0u^{\epsilon_1}\ldots u^{\epsilon_n}x_n\,\,\text{is a reduced word with}\,(\epsilon_n=- 1)\,\text{or}\,(\epsilon_n=1\,\text{and}\, E_B(x_n)=0)\},
$$
and $K_{1}=\overline{\text{Span}}\{\rho(x)\eta\,:\,x=x_0u^{\epsilon_1}\ldots u^{\epsilon_n}\,\,\text{is a reduced word with}\,\epsilon_n=1\}$. Observe that $K_{-1}$ and $K_1$ are orthogonal subspaces. Moreover, since $\rho(b)\eta=\epsilon(b)\eta$ for all $b\in B$, one has $K=K_{-1}\oplus K_1$.

By lemma \ref{iso} (and since $\epsilon\circ\theta=\epsilon$) we have isometries $F_i\,:\, H_i\rightarrow K_i$, for $i=\pm 1$, defined by, for $x=x_0u^{\epsilon_1}\ldots u^{\epsilon_n}$ a reduced word in $P$,
$$F_{-1}(\pi(x)\xi)=\rho(xu)\eta\,\,\text{if}\,\,\epsilon_n=-1\quad\text{and}\quad F_1(\pi(x)\xi)=\rho(x)\eta\,\,\text{if}\,\,\epsilon_n=1.$$
Since $F_{-1}$ and $F_1$ are clearly surjective, they are unitaries. Hence, we get a unitary
$$F=F_{-1}\oplus F_1\,:\,H_{-1}\oplus H_1\rightarrow K_{-1}\oplus K_1.$$

We define the \textit{Julg-Valette operator} $\mathcal{F}\,:\,H\rightarrow K$ by $\mathcal{F}|_{\C\xi}=0$ and $\mathcal{F}|_{H_{-1}\oplus H_1}=F$. Hence, $\mathcal{F}$ is a partial isometry with $\mathcal{F}\mathcal{F}^*=1$ and $\mathcal{F}^*\mathcal{F}=1-p$ where $p$ is the orthogonal projection onto the one dimensional subspace $\C\xi$.

\begin{lemma}\label{compact}
The following holds.
\begin{enumerate}
\item For all $a\in A$, $\mathcal{F}\pi(a)=\rho(a)\mathcal{F}$.
\item $\mathcal{F}\pi(u)-\rho(u)\mathcal{F}$ is a rank one operator with image $\C\rho(u)\eta$.
\item $\mathcal{F}\pi(u^*)-\rho(u^*)\mathcal{F}$ is a rank one operator with image $\C\eta$.
\end{enumerate}
\end{lemma}

\begin{proof}
$1.$ Let $a\in A$. One has $\mathcal{F}\pi(a)\xi=\epsilon(a)\mathcal{F}\xi=0=\rho(a)\mathcal{F}\xi$ and, for $x=x_0u^{\epsilon_1}\ldots u^{\epsilon_n}$ a reduced word,
$$\mathcal{F}\pi(a)\pi(x)\xi=\mathcal{F}\pi(ax)\xi=\left\{\begin{array}{lcl}
\rho(axu)\eta=\rho(a)\mathcal{F}\pi(x)\xi &\text{if}&\epsilon_n=-1,\\
 \rho(ax)\eta=\rho(a)\mathcal{F}\pi(x)\xi &\text{if}&\epsilon_n=1.\end{array}\right.$$
 
$2.$ Let $x=x_0u^{\epsilon_1}\ldots u^{\epsilon_n}$ be a reduced word. If $n\geq 2$ then it is easy to see that $ux$ can be written has a reduced word or a sum of two reduced words that end with $u^{\epsilon_n}$. Hence,
 $$\mathcal{F}\pi(u)\pi(x)\xi=\mathcal{F}\pi(ux)\xi=\left\{\begin{array}{lcl}
\rho(uxu)\eta=\rho(u)\mathcal{F}\pi(x)\xi &\text{if}&\epsilon_n=-1,\\
 \rho(ux)\eta=\rho(u)\mathcal{F}\pi(x)\xi &\text{if}&\epsilon_n=1.\end{array}\right.$$
 If $n=1$ and $x=x_0u^{\epsilon}$. When ($\epsilon=1$) or ($\epsilon=-1$ and $E_B(x_0)=0$) we see that $ux$ can be written as a reduced word that ends with $u^{\epsilon}$. As before, we conclude that $\mathcal{F}\pi(u)\pi(x)\xi=\rho(u)\mathcal{F}\pi(x)\xi$ in this case. Hence, the operator $\mathcal{F}\pi(u)-\rho(u)\mathcal{F}$ vanishes on the subspace $L$ where
 \begin{eqnarray*}
 L^{\bot}&=&\overline{\text{Span}}\{\xi,\pi(x_0u^*)\xi\,:\,x_0\in B\}=\overline{\text{Span}}\{\xi,\pi(u^*\theta(x_0))\xi\,:\,x_0\in B\}\\
 &=&\overline{\text{Span}}\{\xi,\epsilon\circ\theta(x_0)\pi(u^*)\xi\,:\,x_0\in B\}=\C\xi\oplus\C\pi(u^*)\xi.
 \end{eqnarray*}
 Since $(\mathcal{F}\pi(u)-\rho(u)\mathcal{F})(\pi(u^*)\xi)=\mathcal{F}\xi-\rho(u)\eta=-\rho(u)\eta$ and  $(\mathcal{F}\pi(u)-\rho(u)\mathcal{F})\xi=\rho(u)\eta$, this finishes the proof of $2$. The proof of $3$ is similar.
\end{proof}

Since $P$ is the closed linear span of the reduced words and $A$, Lemma \ref{compact} implies that $\mathcal{F}\pi(x)-\pi(x)\mathcal{F}$ is a compact operator for all $x\in P$. Hence, the triple $(\pi,\rho,\mathcal{F})$ defines an element $\alpha\in{\rm KK}(P,\C)$. To prove Theorem \ref{Kamenability}, it suffices to show that $\lambda^*(\alpha)=[\epsilon]$ in ${\rm KK}(P_m,\C)$, where $P_m$ be the maximal $C^*$-algebra of the HNN extension i.e, the maximal HNN extension, and $\epsilon$ is the trivial representation of $P_m$.

Define $\widetilde{K}=K\oplus\C\Omega$, where $\Omega$ is a norm one vector, with the representation $\widetilde{\rho}=\rho\circ\lambda\oplus\epsilon$ of $P_m$. Define the unitary $\widetilde{\mathcal{F}}\,:\,H\rightarrow\widetilde{K}$ by
$$\widetilde{\mathcal{F}}\xi=\Omega\quad\text{and}\quad\widetilde{\mathcal{F}}|_{H_{-1}\oplus H_1}=F.$$
The triple $(\widetilde{\pi}\circ\lambda,\widetilde{\rho},\widetilde{\mathcal{F}})$, where $\widetilde{\pi}=\pi\circ\lambda$, defines an element $\gamma\in{\rm KK}(P_m,\C)$ satisfying $\gamma=\lambda^*(\alpha)-[\epsilon]$. It suffices to show that $(\widetilde{\pi},\widetilde{\rho},\widetilde{\mathcal{F}})$ is homotopic to a degenerated triple.

Define the unitary $v\in\mathcal{B}(\widetilde{K})$ by
$$v\eta=\Omega,\quad v\Omega=\eta,\quad v\rho(x)\eta=\rho(x)\eta\,\,\text{for}\,x\in P\,\text{with}\, E_B(x)=0.$$ 

\begin{lemma}\label{homotopy}
Write $P_m=\langle A,w\rangle$. The following holds.
\begin{enumerate}
\item $\widetilde{\mathcal{F}}\widetilde{\pi}(a)\widetilde{\mathcal{F}}^*=\widetilde{\rho}(a)$ for all $a\in A\subset P_m$.
\item $\widetilde{\mathcal{F}}\widetilde{\pi}(w)\widetilde{\mathcal{F}}^*=\widetilde{\rho}(w)v$.
\item $v\widetilde{\rho}(b)v^*=\widetilde{\rho}(b)$ for all $b\in B$.
\end{enumerate}
\end{lemma}

\begin{proof}
$1.$ Let $a\in A$. One has $\widetilde{\mathcal{F}}\pi(a)\widetilde{\mathcal{F}}^*\Omega=\widetilde{\mathcal{F}}\pi(a)\xi=\epsilon(a)\widetilde{\mathcal{F}}\xi=\epsilon(a)\Omega=\widetilde{\rho}(a)\Omega$. Since $\widetilde{\mathcal{F}}|_{H_{-1}\oplus H_1}=\mathcal{F}|_{H_{-1}\oplus H_1}$ we find, using assertion $1$ of Lemma \ref{compact}, that
$$\widetilde{\mathcal{F}}\pi(a)\widetilde{\mathcal{F}}^*|_K=\rho(a)|_K=\widetilde{\rho}(a)|_K.$$ This concludes the proof of $1$.

$2.$ Since $\widetilde{\pi}(w)=\pi(u)$, it suffices to prove the following.

\begin{itemize}
\item $\widetilde{\mathcal{F}}\pi(u)\widetilde{\mathcal{F}}^*\Omega=\widetilde{\rho}(w)v\Omega$.
\item $\widetilde{\mathcal{F}}\pi(u)\widetilde{\mathcal{F}}^*\eta=\widetilde{\rho}(w)v\eta$.
\item $\widetilde{\mathcal{F}}\pi(u)\widetilde{\mathcal{F}}^*\rho(x)\eta=\widetilde{\rho}(w)v\rho(x)\eta$ for all $x\in P$ such that $E_B(x)=0$.
\end{itemize}
Since $\widetilde{\rho}(w)v\Omega=\widetilde{\rho}(w)\eta=\rho(u)\eta$, the first point follows from the computation:
$$\widetilde{\mathcal{F}}\pi(u)\widetilde{\mathcal{F}}^*\Omega=\widetilde{\mathcal{F}}\pi(u)\xi=F\rho(u)\eta=\rho(u)\eta.$$
Since $w\in P_m$ is an irreducible unitary representation we get $\epsilon(w)=1$ and $\widetilde{\rho}(w)v\eta=\widetilde{\rho}(w)\Omega=\epsilon(w)\Omega=\Omega$. Hence, the second point follows from the computation:
$$\widetilde{\mathcal{F}}\pi(u)\widetilde{\mathcal{F}}^*\eta=\widetilde{\mathcal{F}}\pi(u)F^*\eta=\widetilde{\mathcal{F}}\pi(u)\pi(u^*)\xi=\widetilde{\mathcal{F}}\xi=\Omega.$$
For the last point we separate the different cases. First, observe that, for $x\in P$ with $E_B(x)=0$, one has $\widetilde{\rho}(w)v\rho(x)\eta=\widetilde{\rho}(w)\rho(x)\eta=\rho(u)\rho(x)\eta=\rho(ux)\eta$.

\textbf{Case 1:} \textit{If $x\in A$ and $E_B(x)=0$}. Then, since $uxu^*$ is reduced,
$$\widetilde{\mathcal{F}}\pi(u)\widetilde{\mathcal{F}}^*\rho(x)\eta=\widetilde{\mathcal{F}}\pi(u)F^*\rho(x)\eta
=\widetilde{\mathcal{F}}\pi(u)\pi(xu^*)\xi=F\pi(uxu^*)\xi=\rho(ux)\eta.$$
For the other case i.e. when $E_A(x)=0$, we can assume that $x=x_0u^{\epsilon_1}\ldots u^{\epsilon_n}x_n$ is reduced. We separate again in different cases.

\textbf{Case 2:} \textit{If $x=x_0u^{\epsilon_1}\ldots u^{\epsilon_n}x_n$ is reduced with $\epsilon_n=-1$}. One has 
$$\widetilde{\mathcal{F}}\pi(u)\widetilde{\mathcal{F}}^*\rho(x)\eta=\widetilde{\mathcal{F}}\pi(u)F^*\rho(x)\eta
=\widetilde{\mathcal{F}}\pi(uxu^*)\xi.$$
Since $\epsilon_n=-1$, $uxu^*$ can be written as a reduced word or the sum of two reduced words that end with $u^*$. Hence,
$\widetilde{\mathcal{F}}\pi(u)\widetilde{\mathcal{F}}^*\rho(x)\eta=F\pi(uxu^*)\xi=\rho(ux)\eta$.

\textbf{Case 3:} \textit{$\epsilon_n=1$}. If $x_n\in B$, since $\rho(b)\eta=\epsilon(b)\eta$ we may assume that $x_n=1$. Then,
$$\widetilde{\mathcal{F}}\pi(u)\widetilde{\mathcal{F}}^*\rho(x)\eta=\widetilde{\mathcal{F}}\pi(u)F^*\rho(x)\eta
=\widetilde{\mathcal{F}}\pi(ux)\xi.$$
Since $\epsilon_n=1$, $ux$ can be written as a reduced word or the sum of two reduced words that end with $u$. Hence,
$$\widetilde{\mathcal{F}}\pi(u)\widetilde{\mathcal{F}}^*\rho(x)\eta=F\pi(ux)\xi=\rho(ux)\eta.$$
If $E_B(x_n)=0$ then $\widetilde{\mathcal{F}}\pi(u)\widetilde{\mathcal{F}}^*\rho(x)\eta=\widetilde{\mathcal{F}}\pi(u)F^*\rho(x)\eta
=\widetilde{\mathcal{F}}\pi(uxu^*)\xi$. Since $\epsilon_n=1$ and $E_B(x_n)=0$, $uxu^*$ can be written as a reduced word or the sum of two reduced words that end with $u^*$. Hence,
$$\widetilde{\mathcal{F}}\pi(u)\widetilde{\mathcal{F}}^*\rho(x)\eta=F\pi(uxu^*)\xi=\rho(ux)\eta.$$

$3.$ Let $b\in B$. One has $v\widetilde{\rho}(b)v^*\Omega=v\rho(b)\eta=\epsilon(b)v\eta=\epsilon(b)\Omega=\widetilde{\rho}(b)\Omega$. Moreover,
$$v\widetilde{\rho}(b)v^*\eta=v\widetilde{\rho}(b)\Omega=\epsilon(b)v\Omega=\epsilon(b)\eta=\rho(b)\eta
=\widetilde{\rho}(b)\eta.$$
Eventually, for $x\in P$ with $E_B(x)=0$, one has $E_B(bx)=0$. Hence, $v\widetilde{\rho}(b)v^*\rho(x)\eta=v\rho(bx)\eta=\rho(b)\rho(x)\eta$.
\end{proof}

\textit{End of the proof of Theorem \ref{Kamenability}.} By Lemma \ref{homotopy}, $v\in\widetilde{\rho}(B)'\cap\mathcal{B}(\widetilde{K})$. Let $a\in\widetilde{\rho}(B)'\cap\mathcal{B}(\widetilde{K})$ be the unique self-adjoint element with spectrum $[-\pi,\pi]$ such that $v=e^{ia}$ and define, for $s\in\R$, $v_s=e^{isa}$. It follows that $v_s$ is a continuous one-parameter group of unitaries in $\widetilde{\rho}(B)'\cap\mathcal{B}(\widetilde{K})$. For $s\in\R$ define the unitary $w_s=\widetilde{\rho}(w)v_s\in\mathcal{B}(\widetilde{K})$. Observe that, for all $b\in B$ and all $s\in\R$,
$$w_s\widetilde{\rho}(b)w_s^*=\widetilde{\rho}(w)v_s\widetilde{\rho}(b)v_s^*\widetilde{\rho}(w^*)=\widetilde{\rho}(w)\widetilde{\rho}(b)\widetilde{\rho}(w^*)=\widetilde{\rho}(wxw^*)=\widetilde{\rho}(\theta(x)).$$
By the universal property of $P_m$, for each $s\in\R$, there exists a unique representation $\rho_s$ of $P_m$ on $\widetilde{K}$ such that
$$\rho_s(w)=w_s\quad\text{and}\quad\rho_s(a)=\widetilde{\rho}(a)\,\,\text{for}\,a\in A.$$
The family of triples $x_s=(\widetilde{\pi},\rho_s,\widetilde{F})$, for $s\in\R$, defines an homotopy between $x_0=(\widetilde{\pi},\widetilde{\rho},\widetilde{F})$ and $x_1$ which is degenerate by Lemma \ref{homotopy}.\end{proof}

\end{document}